\documentclass[12pt]{amsart}

\usepackage[margin=1.15in]{geometry}
\usepackage{amssymb, amsmath, wasysym, mathrsfs, mathtools, color, stmaryrd}
\usepackage{graphicx}
\usepackage[all, cmtip]{xy}
\usepackage{url}

\definecolor{hot}{RGB}{65,105,225}

\usepackage[pagebackref=true,colorlinks=true, linkcolor=hot ,  citecolor=hot, urlcolor=hot]{hyperref}

\theoremstyle{plain}
\newtheorem{theorem}{Theorem}[section]
\newtheorem{prop}[theorem]{Proposition}

\newtheorem{cor}[theorem]{Corollary}

\newtheorem{lemma}[theorem]{Lemma}
\newtheorem{thrm}[theorem]{Theorem}

\theoremstyle{definition}

\newtheorem{defn}[theorem]{Definition}
\newtheorem{rmk}[theorem]{Remark}

\newtheorem*{ex*}{Example}

\newcommand\cO{{\mathcal O}}

\newcommand\bQ{{\mathbf{Q}}}
\newcommand\bZ{{\mathbf{Z}}}
\newcommand\bR{{\mathbf{R}}}
\newcommand\bC{{\mathbf{C}}}
\newcommand\bA{{\mathbf{A}}}
\newcommand\bN{{\mathbf{N}}}

\newcommand{\ubul}{{\,\begin{picture}(-1,1)(-1,-3)\circle*{2}\end{picture}\ }}

\DeclareMathOperator{\ehom}{End}

\DeclareMathOperator{\rank}{rank}
\DeclareMathOperator{\Rep}{Rep}
\DeclareMathOperator{\Mat}{Mat}
\DeclareMathOperator{\Aut}{Aut}
\DeclareMathOperator{\Ext}{Ext}
\DeclareMathOperator{\homo}{Hom}

\DeclareMathOperator{\spec}{Spec}

\def\ra{\rightarrow}
\def\lra{\longrightarrow}
\def\SL{{\mathrm{SL}}}
\def\al{\alpha}
\def\ol{\overline}
\def\be{\begin{equation}}
\def\ee{\end{equation}}
\def\xa{\xrightarrow}

\def\eps{\epsilon}
\def\lam{\lambda}
\def\bG{\mathbb{G}}
\def\bGL{{\rm{GL}}}
\def\GL{{\rm{GL}}}
\def\sS{\mathscr{S}}

\newcommand{\R}[2]{{{R}}(#1,#2)}

\title[Rational singularities and moment maps]{Rational singularities, quiver moment maps, and representations of surface groups}

\author{Nero Budur}
\address{KU Leuven, Celestijnenlaan 200B, B-3001 Leuven, Belgium} 
\email{nero.budur@kuleuven.be}


\begin{document}

\date{}

\begin{abstract} We prove using jet schemes that the zero loci of the moment maps for the quivers with one vertex and at least two loops have rational singularities. This implies that the spaces of representations of the fundamental group of a compact Riemann surface of genus at least two have rational singularities. This has consequences for the numbers of irreducible representations of the special linear groups over the integers and over the $p$-adic integers.
\end{abstract}

\maketitle
\tableofcontents

\section{Introduction}

 Let $\Mat(n,\bC)$ be the space of $n\times n$ matrices of complex numbers. We prove:

\begin{thrm}\label{thrmRSQ}
Let $g\ge 2$. The set 
$$
X=\{(x_1,y_1,\ldots,x_g,y_g)\in \Mat(n,\bC)^{\times 2g}\mid [x_1,y_1]+\ldots + [x_g,y_g]=0\},
$$
where $[x,y]=xy-yx$, is a variety with rational singularities for all $n\ge 1$.
\end{thrm}

Throughout this article, a variety over an algebraically closed field is an integral scheme of finite type.

The variety $X$ from Theorem \ref{thrmRSQ} is the zero locus of the moment map of the quiver with one vertex and $g$ loops. In this context, Crawley-Boevey \cite{CB} showed that $X$ with its natural scheme structure is a complete intersection variety for $g\ge 2$.  More recently, Aizenbud-Avni \cite{AA} showed that $X$ has rational singularities for $g\ge 12$. The proof of Theorem \ref{thrmRSQ} here is different than theirs and works for all $g\ge 2$. This the main result of the article, the rest of the results being corollaries. 

Rational singularities are a natural class of mild singularities. For example, the rational double points of surfaces are classified by the ADE diagrams. While the definition is technical, the importance of rational singularities lies in the abundance of applications they entail. We will see that the theorem has several applications, among which is a bound for the number of irreducible representations of $\SL_n(\bZ)$ and of $\SL_n(\bZ_p)$.

The proof of Theorem \ref{thrmRSQ} consists of different steps:
\begin{itemize}
\item We use a criterion due to Musta\c{t}\u{a} \cite{Mus} saying that a locally complete intersection variety of dimension $d$ has rational singularities if and only if for all $m\ge 1$ the dimension of the scheme of $m$-jets passing through the singular locus is smaller than $d(m+1)$.
\item The theorem is proven in terms of quivers. For a quiver with the set of directed arrows $Q$ and a dimension vector $\alpha$ on its set of vertices, we consider the zero locus of the associated moment map
$$
X(Q,\al) = \{(x_a, x_{a^*})_{a\in Q} \mid \sum_{a\in Q}(x_ax_{a^*}-x_{a^*}x_a)=0\}
$$
where $x_a$ and $x_{a^*}$ are matrices representing linear transformations along the arrow $a$ and backwards, along an added reverse arrow $a^*$, between vector spaces of ranks specified by $\al$. 

\item The proof now is inspired by the proof of normality for the Mardsen-Weinstein reductions for quiver representations by Crawley-Boevey \cite{CB-norm}. Namely, we stratify $X(Q,\al)$ by the representation type of the semisimplification of each representation in $X(Q,\al)$.  Applying Luna slice theorem, locally at the semisimple points of each stratum, $X(Q,\al)$ is up to a smooth factor the germ at $0$ of $X(Q',\al')$ for a new quiver and dimension vector $(Q',\al')$ depending only on the representation type corresponding to the stratum. %By Elkik's deformation of rational singularities, to show that $X(Q,\al)$ has rational singularities it is enough to prove that $0$ is a rational singularity for all $X(Q',\al')$ arising from $X(Q,\al)$. 

\item We have a distinguished stratum, $Z(Q,\al)$, consisting of all the representations in $X(Q,\al)$ with semisimplification having the same representation type as $0\in X(Q,\al)$. The stratum $Z(Q,\al)$ should be viewed as the ``most singular" stratum of $X(Q,\al)$. Using Musta\c{t}\u{a}'s criterion we show that: if $\dim Z(Q',\al')< 2p_Q(\al)$ for all $(Q',\al')$ arising as above from $(Q,\al)$, where $p_Q(\al)$ is determined from the combinatorics of the quiver, then $X(Q,\al)$ has rational singularities. The crucial argument is that one can easily bound the dimension of the set of $m$-jets over the trivial representation for any $X(Q,\al)$.

\item The bound $\dim Z(Q,\al)< 2p_Q(\al)$, as well as the property that $X(Q,\al)$ has rational singularities, does not hold for all $(Q,\al)$. We show that the bound holds for all $(Q',\al')$ obtained from the quiver with one vertex and $g\ge 2$ loops. This relies on \cite{CB-norm} as well. 
\end{itemize}

Theorem \ref{thrmRSQ} will allow us to prove:

\begin{theorem}\label{thrmGLRat}
Let $C_g$ be a compact Riemann surface of genus $g\ge 2$. Then $$\homo(\pi_1(C_g),\bGL_n(\bC)) =\{(x_1,y_1,\ldots,x_g,y_g)\in \bGL_n(\bC)^{\times 2g}\mid \prod_{i=1}^gx_iy_ix_i^{-1}y_i^{-1}=I \}$$ is a variety with rational singularities.  
\end{theorem}

By $\pi_1(C_g)$ we denote the fundamental group of $C_g$ based at a fixed point. Simpson \cite{Si} showed that $\homo(\pi_1(C_g),\GL_n(\bC))$ with its natural scheme structure is a normal complete intersection variety for $g\ge 2$.

Reducedness in the case $g=1$ of both Theorems \ref{thrmRSQ} and \ref{thrmGLRat} is an old still open conjecture.

To conclude the second theorem from the first  we use an observation of  Kaledin-Lehn-Sorger \cite{KLS}.  Namely, the local structure of $\homo(\pi_1(C_g),\GL_n(\bC))$ at semisimple representations is modelled by zero loci of moment maps of quivers.

The last theorem implies:

\begin{theorem}\label{thrmSLD}
Let $C_g$ be a compact Riemann surface of genus $g\ge 2$. Then $$\homo(\pi_1(C_g),\SL_n(\bC)) =\{(x_1,y_1,\ldots,x_g,y_g)\in \SL_n(\bC)^{\times 2g}\mid \prod_{i=1}^gx_iy_ix_i^{-1}y_i^{-1}=I \}$$ is a variety with rational singularities.  
\end{theorem}

The proof shows also that the natural scheme structure is reduced, a fact which seems to be missing from the literature in full generality, according to A. Sikora. The reduced scheme structure of $\homo(\pi_1(C_g),\SL_n(\bC))$ was shown to be a complete intersection variety for $g\ge 2$ in \cite{RBC}. 

For $g\ge 12$, Theorem \ref{thrmSLD} is proven  in \cite{AA}. Their proof uses $p$-adic group theory to show rational singularities beyond the trivial representation. The proof we give here is entirely geometric and it applies to $g\ge 2$.

Since the natural scheme structure on the varieties from Theorems \ref{thrmRSQ}, \ref{thrmGLRat}, \ref{thrmSLD} is reduced, and since all these schemes are defined over $\bQ$, by general scheme theory it follows that the $\bQ$-schemes naturally associated to these sets are $\bQ$-varieties with only rational singularities. Here a $\bQ$-variety $V$ is a geometrically irreducible, reduced, separated scheme of finite type over $\bQ$, and $V$ has rational singularities if for any resolution of singularities $f:W\ra V$ over $\bQ$, the natural morphism $\cO_W\ra Rf_*\cO_V$ is a quasi-isomorphism. This generalization has the following consequences by Aizenbud-Avni \cite{AA, AA1}, see  Section \ref{secAAG}.

\begin{theorem}\label{thrmGAM0} 
 Let $n\ge 3$ be an integer. There exist a positive real number $C$ and a natural number $m_0$ such that  for all integers $m\ge m_0$, the number of isomorphism classes of irreducible complex representations of dimension at most $m$ of $\SL_n(\bZ)$ is less or equal to $Cm^2$.
\end{theorem} 
 
\begin{theorem}\label{thrmGAM2}
Let $n\ge 1$ be an integer and $p$ a prime number. There exists a positive real number $C$ such that for all integers $m$, the number of isomorphism classes of continuous irreducible $m$-dimensional complex representations of $\SL_n(\bZ_p)$ is less than $Cm^2$.
\end{theorem}
 
These theorems improve the bound $Cm^{22}$ of \cite{AA,AA1} to the quadratic bound. The theorems follow from the next two, more general results. See  Section \ref{secAAG} for the definition of the abscissa of convergence $\al$ of a group.
 
\begin{theorem}\label{thrmGAM1}
For every non-archimedean local field F containing $\bQ$ and every compact
open subgroup  $\Gamma\subset \SL_n(F)$, the abscissa $\al(\Gamma)<2$.
\end{theorem}

\begin{theorem}\label{thrmGAM3}

 Let $n$ be a positive integer. Then there exists a finite set $S$ of prime numbers such that for every global field $k$ of characteristic not in $S$ and every finite set $T$ of places of $k$ containing all Archimedean ones,
 $$
 \al(\SL_n(O_{k,T}))\le 2,
 $$
 where 
 $$
 O_{k,T}=\{ x\in k\mid \forall v\not\in T, ||x||_v\le 1 \},
 $$
 with the exception of the case $n=2$ and $k$ equal to $\bQ$ or an imaginary quadratic extension of $\bQ$. In particular, 
 $$
 \al(\SL_n(\bZ))\le 2\, \quad \text{ for }n\ge 3.
 $$
\end{theorem}
In \cite{AA, AA1}, the abscissae in these theorems  have been bounded above by 22, instead of 2.

Another consequence of Theorem \ref{thrmSLD} is:

\begin{theorem}\label{thrmGAM4}

Let $g\ge 2$, $n\ge 1$ be integers. The map of $\bQ$-varieties
$$
\SL_n^{\times 2g}\ra \SL_n,\quad (x_1,y_1,\ldots,x_g,y_g)\mapsto \prod_{i=1}^g x_iy_ix_i^{-1}y_i^{-1}$$
is flat and the inverse image of every $z\in\SL_n(\bC)$ is a variety with rational singularities.
\end{theorem}

This theorem has been proven in \cite{AA} for $g\ge 12$.

 Theorems \ref{thrmGAM0}-\ref{thrmGAM4}, together with Theorem \ref{thrmSLD} which implies them, were claimed as stated here in \cite{BZo}. However, \cite{BZo} contains a gap in an argument involving $p$-adic group theory which has not been fixed yet, see \cite{BZo-er}. Theorem \ref{thrmSLD} here implies that all the claims of \cite[\S 1]{BZo} nevertheless hold, except \cite[Proposition 1.5]{BZo} which remains open.

\smallskip
In connection with the last theorem, we can add:

\begin{theorem} Let $g\ge 2$, $n\ge 1$ be integers.

(a) The map of $\bQ$-varieties
$$
\Mat(n)^{\times 2g}\ra \Mat(n),\quad (x_1,y_1,\ldots,x_g,y_g)\mapsto \sum_{i=1}^g [x_i,y_i]$$
it is flat and the inverse image of every point $z\in \Mat(n,\bC)$ in some open neighborhood of $0$ is a variety with rational singularities.

(b)  The map of $\bQ$-varieties
$$
\GL_n^{\times 2g}\ra \GL_n,\quad (x_1,y_1,\ldots,x_g,y_g)\mapsto \prod_{i=1}^g x_iy_ix_i^{-1}y_i^{-1}$$
has the following property over a open neighborhood of $I\in \GL_n$: it is flat and the inverse image of every $z\in\GL_n(\bC)$ is a variety with rational singularities.
\end{theorem} 

The flatness in (a) is a particular case of \cite[Theorem 1.1]{CB} for the quiver with one vertex and $g$ loops. Hence (a) follows from Theorem \ref{thrmRSQ} here and the open nature of rational singularities in flat families \cite[Th\'eor\'eme 4]{El}. It is showed in \cite[2.3.1]{AA} that (a) implies (b).

Finally, the Isosingularity Principle of Simpson \cite[Theorem 10.6]{Si} applies to transfer the rational singularities property from representations spaces of fundamental groups to the de Rham and Dolbeault analogs, we refer to \cite{Si} for the precise definitions:

\begin{theorem}
Let $g\ge 2$, $n\ge 1$ be integers. Let $G$ be $\GL_n$ or $\SL_n$. Let $C_g$ be a compact Riemann surface and $x\in C_g$ a fixed point. Let $R_{DR}(C_g,G)$ be the fine moduli space of principal $G$-bundles with integrable connection and a frame over $x$. Let $R_{Dol}(C_g,G)$ be the fine moduli space of principal Higgs bundles for the group $G$ with a frame over $x$, which are semistable with vanishing rational Chern classes. Then  $R_{DR}(C_g,G)$ and $R_{Dol}(C_g,G)$ have rational singularities over $\bC$.
\end{theorem}

Since GIT quotients of rational singularities are also rational singularities, one obtains as corollaries of our results that for $g\ge 2$ and $G=\GL_n(\bC), \SL_n(\bC)$ the following schemes have rational singularities: the Mardsen-Weinstein reduction $X\sslash G$, the moduli of local systems $M_{B}(C_g,G)$, the moduli of bundles with integrable connections $M_{DR}(C_g,G)$, and the moduli of Higgs bundles with vanishing Chern classes $M_{Dol}(C_g,G)$. This is already known from the fact that all these spaces have symplectic singularities, see \cite[Corollary 8.4]{CB-norm}, \cite{BS}, \cite{Tir}, and our results give a different explanation.

In Section \ref{s2}, we recall the description of the \'etale slices for zeros of moment maps of quivers due to Crawley-Boevey, and prove some preliminary results on the quiver $Q$. In Section \ref{s3}, the heart of the article, we prove Theorem \ref{thrmRSQ}. In Section \ref{s4} we prove Theorem \ref{thrmGLRat} using deformation theory. In Section \ref{s5} we prove Theorem \ref{thrmSLD}. In Section \ref{secAAG} we address Theorems \ref{thrmGAM0}-\ref{thrmGAM4}.

 {\it Acknowledgement.}  We thank  M. Zordan, R. Zhao, and the referees for useful comments. The author was partly sponsored by the Methusalem grant METH/15/026 from KU Leuven, and by the research projects G0B2115N, G0F4216N, G097819N from the Research Foundation of Flanders.

\section{Quiver moment maps}\label{s2}

We work over a fixed algebraically closed field $k$ of characteristic zero. A quiver $Q$ is a set of directed arrows between a set $I$ of vertices. The representations of $Q$ of dimension vector $\al\in\bN^I\setminus\{0\}$ are the elements of
$$
\Rep(Q,\al)=\bigoplus_{a\in Q} \Mat (\al_{h(a)}\times \al_{t(a)},k),
$$
where $h(a)$ is the head vertex of the arrow $a$, and $t(a)$ is the tail. The group
$$
G(\al)=\left(\prod_{i\in I}\GL_{\al_i}(k) \right)/k^*
$$
acts by conjugation.

A morphism between two quiver representations $x\in \Rep(Q,\al)$ and $y\in\Rep(Q,\beta)$ is an element 
$$f\in \bigoplus_{i\in I}\Mat(\al_i\times \beta_i, k)$$
such that $fx=yf$. That is, $f_{h(a)}x_a=y_af_{t(a)}$ for all $a\in Q$.  

The orbits of $G(\al)$ on $\Rep(Q,\al)$ correspond to isomorphism classes of representations. 

The category of representations of a quiver $Q$ is equivalent to the category of left $kQ$-modules, where $kQ$ is the path algebra of $Q$. 

%For $x\in \Rep(Q,\al)$, one defines the groups $\Ext^i(x,x)$ as the extension groups of the associated $kQ$-module.

A representation in $\Rep(Q,\al)$ is called simple if the associated  $kQ$-module is simple, and semisimple if it is a direct sum of simple representations.

Let $p_Q(\al)=1-\langle\al,\al\rangle_Q$, where 
$$
\langle \al,\beta \rangle_Q = \sum_{i\in I}\al_i\beta_i-\sum_{a\in Q}\al_{t(a)}\beta_{h(a)}.
$$
is the  pairing on $\bZ^I$ associated to $Q$. We also define $(\al,\beta)_Q=\langle \al,\beta \rangle_Q+\langle \beta,\al \rangle_Q$. 

The double of $Q$ is the quiver $\ol{Q}$ obtained by adjoining a reverse arrow $a^*$ for each $a\in Q$. Note that
$$
\dim \Rep(\ol{Q},\al) = 2(\al\cdot\al -1 +p_Q(\al)).
$$

There is a $G(\al)$-equivariant moment map
$$
\mu:\Rep(\ol{Q},\al)\lra \ehom(\al)=\bigoplus_{i\in I}\Mat (\al_i,k)
$$
defined by
$$
x\mapsto \sum_{a\in Q}[x_a,x_{a^*}]
$$
where $[x,y]=xy-yx$. We denote the zero locus  by
$$
X(Q,\al) = \mu^{-1}(0)
$$
and we consider it as a closed subscheme; it does not depend on the orientation of the arrows of $Q$ (see \cite[Lemma 2.2]{CBH}). The underlying reduced subscheme of $X(Q,\al)$ is irreducible, but not always a normal variety; see \cite[\S 1]{CB-norm}.

The affine quotient
$$
M(Q,\al) = X(Q,\al)\sslash G(\al)
$$
parametrises isomorphism classes of semisimple representations in $X(Q,\al)$, or, equivalently, the closed orbits of $G(\al)$ on $X(Q,\al)$. Denote by $$q:X(Q,\al)\ra M(Q,\al)$$ the quotient morphism. Every fiber of $q$ contains a unique closed orbit.

We will use later the following results of Crawley-Boevey from Theorem 1.2, Theorem 1.3, Corollary 1.4, Lemma 6.5, and Proof of Corollary 1.4 on p289 of \cite{CB}:

\begin{theorem}\label{thrmCBd} (Crawley-Boevey) If $X(Q,\al)$ contains a simple representation in $\Rep(\ol{Q},\al)$, then:

(a) $X(Q,\al)$ is a reduced and irreducible  complete intersection of dimension $\al\cdot \al-1+2p_Q(\al)$, 

(b) the general element of $X(Q,\al)$ is a simple representation,

(c) the dimension of $M(Q,\al)$ is $2p_Q(\al)$,

(d) $p_Q(\al)>0$ if and only if $M(Q,\al)$ contains an open dense subset of isomorphism classes of simple representations.

(e) The simple representations in $X(Q,\al)$ are smooth points.

\end{theorem}

Given a semisimple representation $x\in X(Q,\al)$ of $\ol{Q}$, we can decompose it  into its simple components
$$
x\simeq x_1^{\oplus e_1}\oplus\cdots\oplus x_r^{\oplus e_r},
$$
where  $x_j$ are non-isomorphic simple representations. If $\beta^{(j)}$ is the dimension vector of $x_j$, one says that {\it $x$ has representation type}
$$
\tau=(e_1,\beta^{(1)};\cdots;e_r,\beta^{(r)}),
$$
and that $\tau$ {\it occurs as a representation type in} $X(Q,\al)$. Set $$e=(e_1,\ldots ,e_r).$$ Then $G(e)$ is a conjugate of  $G(\al)_x$, the stabilizer  of $x$ under the action of $G(\al)$, for a suitable embedding of $G(e)$ in $G(\al)$,  \cite[Theorem 2]{LP}. One obtains a finite stratification into disjoint irreducible locally closed subsets $$M(Q,\al)=\coprod_\tau M(Q,\al)_\tau,$$ where $M(Q,\al)_\tau$ is the subset of isomorphism classes of semisimple representations of type $\tau$, see \cite[\S 11]{CB}. This is the same as the stratification according to conjugacy classes of stabilizers of the semisimple representations, see \cite[Theorem 2]{LP} and \cite[Theorem 5.4]{Ma}:

\begin{thrm}\label{thrmSTST} (i) The assignments $\tau \mapsfrom x \mapsto G(\al)_x$ give a one-to-one correspondence between the set of representation types $\tau$ occurring in $X(Q,\al)$ and the set of conjugacy classes of stabilizers of semisimple points in $X(Q,\al)$.

(ii) For $\tau$ and $\tau'$ representations types corresponding to conjugacy classes of stabilizers $H$ and $H'$, respectively, define $$\tau'\le \tau \iff H \text{ is conjugate to a subgroup of }H'.$$  Then
$$\ol{M(Q,\al)_\tau}=\coprod_{\tau'\le \tau} M(Q,\al)_{\tau'}.$$   
\end{thrm}

By Crawley-Boevey, the  local structure at a semisimple representation of $X(Q,\al)$ reduces to that at the origin with respect to a modified quiver and dimension vector. To state this properly, we need to recall some definitions.

\begin{defn}
Let $G$ be a reductive group acting on an affine variety $X$. Let $H$ be a reductive subgroup of $G$. Let $S$ be a locally closed affine $H$-invariant subvariety of $X$. One defines
$
G\times_HS$
and a $G$-equivariant morphism
$$
G\times_HS \ra X
$$
as follows. Define an action of $H$ on $G\times S$ by $h(g,s)=(gh^{-1},hs)$. Then the map $G\times S\ra X$
sending $(g,s)$ to $gs$ is $H$-equivariant, where $H$ acts trivially on $X$. Then
$$
G\times_HS=(G\times S)\sslash H \ra X
$$
is the associated morphism of affine quotients. Then $G\times_H S$ is an \'etale-locally trivial fibration with fiber type $S$ and base the homogeneous space $G/H$ isomorphic as a variety with any orbit of the action by $H$ on $G$ via right translations. Taking further quotients by the $G$-action, one has a morphism
$$
(G\times_HS)\sslash G = S\sslash H \ra X\sslash G,
$$
cf. \cite[I.3]{Lu}.
\end{defn}

\begin{defn}
Let $G$ be a reductive group acting on affine varieties $X$ and $Y$, and let $f:X\ra Y$ be a $G$-equivariant morphism. We say $f$ is {\it strongly \'etale} if 

- $f/G: X\sslash G\ra Y\sslash G$ is \'etale, and

- $f$, $f/G$, and the quotient morphisms induce a $G$-isomorphism $X\simeq Y\times_{Y\sslash G}(X\sslash G)$. 
\end{defn}

\begin{prop}\label{propDRE}$\rm{(}$\cite[Proposition 4.15]{Dr}$\rm{)}$
If $f:X\ra Y$ is a strongly \'etale $G$-morphism then:

- $f$ is \'etale and surjective,

- for every $u\in X\sslash G$, $f$ induces an isomorphism from the inverse image of $u$ in $X$  to the inverse image of $f(u)$ in $Y$.

- for every $x\in X$, the restriction of $f$ to $Gx$ is injective, and $Gx$ is closed if and only if $Gf(x)$ is closed.
\end{prop}

\begin{lemma}\label{propSE}
If $f:X\ra Y$ is a strongly \'etale $G$-morphism and $x\in X$, then
$G_x=G_{f(x)}$.
\end{lemma}
\begin{proof}
Since $f$ is $G$-equivariant, $G_x\subset G_{f(x)}$. Conversely, if $g\in G_{f(x)}$, then $f(x)=gf(x)=f(gx)$. By the above proposition, $f$ is injective on orbits. Hence $x=gx$, and $g\in G_x$.
\end{proof}

\begin{defn}\label{defEt} Let $G$ be a reductive group acting on an affine variety $X$. Let $x\in X$ be a point with closed orbit. The stabilizer $G_x$ of $x$ is a reductive subgroup of $G$. An {\it \'etale slice} is a $G_x$-invariant locally closed affine subvariety $S$ of $X$ containing $x$  such that  the induced $G$-equivariant morphism
$$
\psi:G\times_{G_x}S\ra X
$$
is strongly \'etale onto its image which is a $G$-saturated affine open subset $U$ of $X$.
\end{defn}

Recall that a $G$-saturated set is the inverse image of a subset of $X\sslash G$ under the quotient morphism $X\ra X\sslash G$.

By the main result of Luna \cite{Lu} \'etale slices always exist. Since the orbit $G(\al)x$ is isomorphic to the homogeneous space $G(\al)/G(\al)_x$ of the action of $G(\al)_x$ on $G(\al)$ by right translations, $\psi$ induces locally in the \'etale topology an  isomorphism  
$$(X,x) \simeq (G(\al)x,x)\times (S,x).
$$

\begin{defn}\label{defCB}
Given a semisimple representation $x\in X(Q,\al)$ of representation type $$\tau=(e,\beta)=(e_1,\beta^{(1)};\cdots;e_r,\beta^{(r)}),$$ define a new quiver $Q_\tau$ with $r$ vertices and whose double $\ol{Q_\tau}$ has $2p_Q(\beta^{(i)})$ loops at vertex $i$, and $-(\beta^{(i)},\beta^{(j)})_Q$ arrows from $i$ to $j$ if $i\ne j$.
\end{defn}

\begin{thrm}\label{thrmELS}$\rm{(}$\cite[\S 4]{CB-norm}$\rm{)}$
Let $x\in X(Q,\al)$ be a semisimple representation of type $\tau=(e,\beta)$. There exists  a morphism
$$
f: S \ra X(Q_\tau,e)
$$
from an \'etale slice $S$ for $X(Q,\al)$ at $x$, sending $x$ to $0$, such that $f$ is equivariant via the canonical isomorphism $G(\al)_x\simeq G(e)$, and the restriction of $f$ is strongly \'etale from an open $G(\al)_x$-saturated neighborhood of $x$ onto an open $G(e)$-saturated neighborhood of $0$.
\end{thrm}
\begin{proof} 
In the proof of \cite[Lemma 4.4]{CB-norm}, and with notation from there, the lemma from \cite[p95]{Lu} is applied to $\phi:U\ra V$. The conclusion from this lemma is stronger than as stated by Crawley-Boevey: namely not only $\phi'/G:U'\sslash G\ra V'\sslash G$ \'etale, but also $\phi':U'\ra V'$ is an \'etale, and $U'\ra V'\times_{V'\sslash G}U'\sslash G$ is an isomorphism. Since $U'$ is  an open $G$-saturated subset of $G\times_{G_x}C'$ (where $C'$ denotes the set $\mu_x^{-1}(0)\cap \nu^{-1}(0)$), by \cite[Lemme, p87]{Lu} one has that $U'=G\times_{G_x}S'$ for some open subset $S'$ of $C'$. Define $S$ to be the image of $S'$ under the translation map $C'\ra\mu^{-1}(0)$ sending $c$ to $x+c$. The remaining properties making $S$ an \'etale slice are already stated in \cite[Lemma 4.4]{CB-norm}.

The morphism $f$  (in our notation) is constructed in \cite[Lemmas 4.8 and 3.3]{CB-norm}. Note that by applying the Fundamental Lemma of \cite[p94]{Lu} to the morphism $\nu^{-1}(0)\ra W$ in the proof of \cite[Lemma 4.8]{CB-norm} (instead of \cite[Lemme 3, p93]{Lu}), one obtains the stronger conclusion that $f$ is  strongly \'etale locally $x$.
\end{proof}

A direct implication of this theorem is:

\begin{cor}\label{propSim0} Let $Q$ be a quiver with dimension vector $\al$. Let $x\in X(Q,\al)$ be a semisimple representation  of type $\tau=(e,\beta)$.  The morphism $f$ from Theorem \ref{thrmELS} induces locally in the \'etale topology  isomorphisms 
$$
(X(Q,\al),x)\simeq (G(\al) x, x)\times (X(Q_\tau,e), 0)
$$
and
$$
(M(Q,\al), q(x)) \simeq (M(Q_\tau,e),q_\tau(0)),
$$
where $q_\tau:X(Q_\tau,e)\ra M(Q_\tau,e)$ is the quotient morphism.
\end{cor}

Next, we show that the process of defining new quivers with dimension vectors via Definition \ref{defCB} does not introduce anything new after repeating it in the following sense:

\begin{prop}\label{prop1T1} For a quiver $Q$ with dimension vector $\al$, let $\tau=(e,\beta)$ be a  representation type occurring in $X(Q,\al)$. 

(i) There is a one-to-one correspondence compatible with the partial order from Theorem \ref{thrmSTST} between the representation types $\tau''\ge \tau$ occurring in $X(Q,\al)$ and the representation types $\tau'$ occurring in $X(Q_\tau,e)$.

(ii) Moreover, if $\tau''=(e'',\beta'')$ corresponds to $\tau'=(e',\beta')$, then
$$
X((Q_\tau)_{\tau'}, e') = X(Q_{\tau''}, e'').
$$
\end{prop}
\begin{proof} Let $X=X(Q,\al)$, $M=M(Q,\al)$, and $G=G(\al)$. Let $x$ be a semisimple point in $X$ of representation type $\tau$, and let $U$ be an open saturated affine neighborhood of $x$. We can then localize the description of
$$
\sS:=\{\tau''\text{ occurs in }X\mid \tau''\ge \tau\}, 
$$
that is, we claim that
$$
\sS =\{\tau''\text{ occurs in }U\mid \tau''\ge \tau\}.
$$
Indeed, the inclusion $\supset$ is obvious, and for the other inclusion suppose that $\tau''\ge \tau$ but $\tau''$ does not occur in $U$. Since $U$ is saturated, $q(U)$ is open in $M$ and intersects $M_\tau$ non-trivially, with $M_\tau$ as in Theorem \ref{thrmSTST}. By assumption, the closure $\ol{M_{\tau''}}$ is contained in the closed subset $M\setminus q(U)$ of $M$. However $q(U)\cap M_\tau\subset M_\tau\subset \ol{M_{\tau''}}$ by Theorem \ref{thrmSTST}. This is a contradiction. This localization principle will be used a few times in this paper.

By Theorem \ref{thrmSTST}, there is a one-to-one correspondence between  the representation types $\tau''\ge \tau$ occurring in $U$ and the conjugacy classes of stabilizers of semisimple points in $U$ conjugate to a subgroup of $H:=G(\al)_x$.

  Let $S$ be an \'etale slice at $x$ as in Theorem \ref{thrmELS}. We take from now on $U$ to be the open affine $G$-saturated neighborhood of $x$ in $X$ obtained as the image of 
 $$\psi: G\times_{H}S\ra X$$
which is strongly \'etale onto $U$, with notation as in  Definition \ref{defEt}. 

By \cite[Proposition 5.5]{Dr}, for every semisimple $x''$ in $U$, the stabilizer $H''$ of $x''$ is the conjugate in $G$ of a subgroup of $H$. Together with the previous remark, we have so far that $\sS$ is into one-to-one correspondence with the $G$-conjugacy classes of stabilizers of semisimple points of $U$.

Since $(G\times_{H}S)\sslash G=S\sslash H$, the $G$-conjugacy classes of stabilizers of semisimple points of $U$ are into one-to-one correspondence with the $H$-conjugacy classes of stabilizers at the points of $S$ with closed $H$-orbit, by \cite[Proposition 4.9 -3]{Dr}.

Let $W$ be an open $H$-saturated affine neighborhood of $x$ in $S$. Since $W\sslash H\subset S\sslash H$ is open, it is \'etale over its image under $\psi/G$ in $M$ which is also open in $M$. This allows us to transfer the localization principle above from $M$ to $S$, that is, to conclude that the $H$-conjugacy classes of stabilizers at the points of $S$ with closed $H$-orbit are the same as the conjugacy classes of stabilizers at the points of $W$ with closed $H$-orbit.

Consider now the morphism $f:S\ra X(Q_\tau,e)$ from Theorem \ref{thrmELS}. Let $W\subset S$ and $V\subset X(Q_\tau,e)$ be open, saturated for the action of $H\simeq G(e)$,  affine neighborhoods of $x$ and $0$, respectively, for which the restriction $f_{|W}:W\ra V$ is strongly \'etale. By Lemma \ref{propSE}, there is a one-to-one correspondence between the conjugacy classes of $H$-stabilizers at the points of $W$ with closed $H$-orbit and the  conjugacy classes of $G(e)$-stabilizers at the points of $V$ with closed $G(e)$-orbit, i.e. semisimple points of $X(Q_\tau,e)$ lying in $V$. By the localization principle from above applied to $X(Q_\tau,e)$, this is further into one-to-one correspondence with the conjugacy classes of the stabilizers of the semisimple points of $X(Q_\tau,e)$. Finally, applying Theorem \ref{thrmSTST}, this proves (i).

For part (ii), write $$\tau=(e,\beta)=(e_i,\beta^{(i)})_{1\le i\le r}$$ and $$\tau'=(e',\beta')=(e'_j,(\beta')^{(j)})_{1\le j\le r'}$$ with $\beta^{(i)}\in\bN^I$ and $(\beta')^{(j)}=((\beta')_1^{(j)},\ldots, (\beta')^{(j)}_r)   \in \bN^r$. Then 
$$
\nu=(e'_j,\sum_{i=1}^r(\beta')^{(j)}_i\cdot \beta^{(i)})_{1\le j\le r'}
$$
is a semisimple representation type $\ge \tau$ in $\Rep(\ol{Q},\al)$, by combinatorial description of the partial order $\ge$ on representation types from \cite[Theorem 2]{LP}. Moreover, $G(e)$ is conjugate to a subgroup of the $G$-stabilizer of $x$ in $X$, and $G(e')$ is conjugate to a subgroup of the $G(e)$-stabilizer of $0$ in $X(Q_\tau,e)$, using the embeddings $G(e)\subset G$ and $G(e')\subset G(e)$ from {\it loc. cit}. Hence, by the correspondence from  part (i), $\nu$ corresponds to $\tau$ and so it occurs also in $X$. Therefore $\nu=\tau''$. With the explicit description of $\tau''$ in hand, it is now straight-forward to check that $
X((Q_\tau)_{\tau'}, e') = X(Q_{\tau''}, e'')
$ directly from Definition \ref{defCB}.
\end{proof}

\begin{prop}\label{propSim} Let $Q$ be a quiver with a dimension vector $\al$. Let $\tau=(e,\beta)$ be a representation type occurring $X(Q,\al)$. 

(i) If $X(Q,\al)$ contains a simple representation, then $X(Q_\tau,e)$ also contains a simple representation. 

(ii) Moreover, in this case
$$p_{Q_\tau}(e)=p_Q(\al).$$
\end{prop}
\begin{proof} Since simple representations correspond to semisimple points with trivial stabilizers, Theorem \ref{thrmSTST} implies that part (i) is a particular case of Proposition \ref{prop1T1}.

In presence of simple representations, by Theorem \ref{thrmCBd} we have then
$$
\dim M(Q,\al) = 2p_Q(\al)\quad\text{ and }\quad\dim M(Q_\tau,e)= 2p_{Q_\tau}(e).
$$
and both are irreducible varieties. Part (ii) then follows from the \'etale-local isomorphism
$$
(M(Q,\al),q(x))\simeq (M(Q_\tau,e),q_\tau(0))
$$ 
from Corollary \ref{propSim0}. 
\end{proof}

\begin{defn}\label{defFZ} Let $Q$ be a quiver with a dimension vector $\al$. We denote the locus in $X(Q,\al)$ fixed by the action of $G(\al)$ by
$$
F(Q,\al) := X(Q,\al)^{G(\al)}.
$$
\end{defn}

\begin{lemma}\label{lemFix} Let $Q$ be a quiver with a dimension vector $\al $. Then:

(i) $F(Q,\al)$ is the set of $x\in \Rep(\ol{Q},\al)$ such that $x_a$ and $x_{a^*}$ are scalar matrices when $a\in Q$ is a loop at a vertex, and $x_a$ and $x_{a^*}$ are 0 otherwise.

(ii) $F(Q,\al)$ is the set of semisimple points $x\in X(Q,\al)$ with representation type $(\al_i,\eps_i)_{i\in I}$, where $\eps_i$ are the standard coordinate vectors on $\bZ^I$.
%(b) Let $\tau=(e_j,\beta^{(j)})_{1\le j\le r}$ be a representation type occuring in $X(Q,\al)$ different than $(\al_i,\eps_i)_{i\in I}$. Then $$\sum_{i\in I}\al_i^2 > \sum_{j=1}^re_j^2.$$
\end{lemma}
\begin{proof}
(i) If $x$ is fixed by all $g\in G(\al)$, then $x_a$ is a square matrix equal to all its conjugates if $a$ is a loop, and it is a matrix equal to all the matrices similar to it, if $a$ is not a loop. The same for $a^*$. Clearly this locus is in the zero locus of the moment map.

(ii) Since $(\al_i,\eps_i)_{i\in I}$ is the representation type of the zero representation in $X(Q,\al)$, this is a particular case of Theorem \ref{thrmSTST}. 
%(b) This follows from the equalities $$\dim G(\al) = \al\cdot \al -1\quad\text{ and }\quad \dim G(e)=e\cdot e -1,$$together with the isomorphism from Theorem \ref{thrmELS} between $G(e)$ and the stabilizer in $G(\al)$ of any semisimple representation $x$ in $X(Q,\al)$. The assumption is equivalent with the stabilizer of $x$ being a closed subvariety strictly contained in $G(\al)$, and hence of stricly smaller dimension.
\end{proof}

%The total dimension of a representation having $\tau$ as representation type is $\al$, so $\al=\sum_j e_j\beta^{(j)}$. In particular, $$\sum_i\al_i^2 \ge \sum_i \sum_je_j^2(\beta^{(j)}_i)^2  = \sum_je_j^2 \Vert\beta^{(j)}\Vert^2, $$ where $\Vert.\Vert$ is the usual norm. For all $j$,  $\Vert \beta^{(j)}\Vert \ge 1$ since dimension vectors are by definition non-zero. By assumption there exists $j_0$ such that $\beta^{(j_0)}$ is not a standard coordinate vector. So $\Vert \beta^{(j_0)}\Vert > 1$, and, since all $e_j$ are positive, this implies that $$ \sum_je_j^2 \Vert\beta^{(j)}\Vert^2 > 1. $$

\begin{defn}\label{defFZa} Let $Q$ be a quiver with a dimension vector $\al$. We denote by $Z(Q,\al)$ the subset of $X(Q,\al)$ consisting of all points with orbit closure containing a fixed point under the action of $G(\al)$. That is,
$$
Z(Q,\al) := q^{-1}(q(F(X,\al)))
$$
where $q:X(Q,\al)\ra M(Q,\al)$ is the affine quotient morphism.
\end{defn}

\begin{lemma}\label{propEXX}
Let $Q$ be any quiver with a dimension vector $\al$. Assume that 
\be\label{eqClass}
\dim q^{-1}(q(0)) < 2(p_Q(\al) - \#(\text{loops in } Q)).
\ee
Then $\dim Z(Q,\al)<2p_Q(\al)$.
\end{lemma}
\begin{proof} Let $X, Z, F$ be $X(Q,\al), Z(Q,\al), F(Q,\al)$, respectively. Since the first \'etale-local isomorphisms of Corollary \ref{propSim0} actually holds above  open saturated affine neighborhoods as provided by Theorem \ref{thrmELS},  we have  for $x\in F$ 
$$
\dim q^{-1}(q(x)) = \dim q^{-1}(q(0))
$$ 
by Proposition \ref{propDRE}. Therefore  $\dim Z\le \dim F + \dim q^{-1}(q(0))$. %For every point $y$ in $Z\setminus F$, there exists an isomorphism between an open neighborhood of $y$ in $Z$ and an open subset of $U\cap Z$. Indeed, one can restrict to $Z$ any isomorphism from a small open neighborhood of $y$ in $X$ via the group action sending $y$ to a point in $Z$ arbitrarily close to $F$. Thus $\dim Z=\dim U\cap Z$, and $$\dim Z\le \dim F + \dim q^{-1}(q(0)).$$
Lemma \ref{lemFix} implies that $\dim F$ is twice the number of loops in $Q$. Together with (\ref{eqClass}) one obtains that $\dim Z< 2p_Q(\al)$.
\end{proof}

\begin{rmk} We will prove below in Proposition \ref{propNil} that (\ref{eqClass}) holds for $(Q_\tau,e)$ with $\tau=(e,\beta)$ arising from semisimple, but not simple, representations in the zero locus of the moment map for the quiver with one vertex and $g\ge 2$ loops. In order to extend Theorem \ref{thrmRSQ} and Theorem \ref{thrmExtend} below on rational singularities of $X(Q,\al)$ to a bigger class of quivers and dimension vectors, it would be helpful to understand for which other $(Q,\al)$ the inequality (\ref{eqClass}) holds.
\end{rmk}

\begin{defn}\label{defnStr} Let $Q$ be a quiver with a dimension vector $\al $. Let  $M_\tau$ be the subset of $M=M(Q,\al)$ consisting of isomorphism classes of semisimple representations of type $\tau$, as in Theorem \ref{thrmSTST}. We define 
$$X_\tau :=q^{-1}(M_\tau),$$
that is, $X_\tau$ is the locally closed subset of $X$ consisting of points $x$ with orbit closure containing a semisimple representation of type $\tau$. For example, $X_\tau=Z(Q,\al)$ if $\tau=(\al_i,\eps_i)_{i\in I}$, by Lemma \ref{lemFix}.
\end{defn}

\begin{prop}\label{propStrat} Let $Q$ be a quiver with a dimension vector $\al $. Let $X=X(Q,\al)$ and let $X_\tau$ be as above, for a representation type $\tau=(e,\beta)$ occurring in $X$. The isomorphism from Corollary \ref{propSim0}  restricts locally in the \'etale topology to an isomorphism
$$
(X_\tau, x)\simeq (G(\al)x,x)\times (Z(Q_\tau,e), 0).
$$
\end{prop}
\begin{proof}
Let $\tau'=(e_i,\eps_i)_{1\le i\le r}$, so that $\tau'$ is the representation type of $0\in X(Q_\tau,e)$ by Lemma \ref{lemFix}. Then by Proposition \ref{prop1T1}, to $\tau'$ corresponds the representation type $\tau''=\tau$  in $X$. Hence the  \'etale-local isomorphism $(M, q(x))\simeq (M(Q_\tau,e), q_\tau(0))$ from Corollary \ref{propSim0} restricts to an \'etale-local isomorphism $(M_\tau, q(x))\simeq (M(Q_\tau,e)_{\tau'}, q_\tau(0))$ with $M=M(Q,\al)$. The conclusion now follows from the first isomorphism from Corollary  \ref{propSim0}, via the strongly \'etale property for $f$ from which this isomorphism was derived, together with the identification $Z(Q_\tau,e)=q_\tau^{-1}(M(Q_\tau,e)_{\tau'})$.
\end{proof}

We recall next a bound on the size of nilpotent cones available for all quivers due to Crawley-Boevey. Afterwards, we show that this bound is particularly useful for the quiver with one vertex and $g\ge 2$ loops. 

\begin{defn}\cite[\S 6]{CB-norm} Let $Q$ be a quiver. Let $T_1,\ldots , T_r$ be a collection of non-isomorphic simple representations of $\ol{Q}$ of dimensions $\beta^{(1)},\ldots, \beta^{(r)}$. Assume that all $T_i$ lie in the zero locus of the moment map (for their corresponding dimension), that is $T_i\in X(Q,\beta^{(i)})$ for all $1\le i\le r$. Let $M$ be in $X(Q,\al)$. We say that $M$ has {\it top-type} $(j_1,m_1;\ldots,j_h,m_h)$ with respect to $T_1,\ldots, T_r$ if:
\begin{itemize}
\item $h\ge 0$, $j_1,\ldots ,j_h$ are integers in $\{1,\ldots, r\}$,  $m_1,\ldots, m_h$ are positive integers, and
\item there is a filtration
$$
0=M_0\subset M_1\subset\ldots\subset M_h=M
$$
by sub-representations such that $M_s/M_{s-1}\simeq T_{j_s}^{\oplus m_s}$ and $\dim \homo (M_s,T_{j_s})=m_s$ for each $s$.
\end{itemize} 

If $(j_1,m_1;\ldots,j_h,m_h)$ is a top-type, and $1\le s\le h$, define $z_s$ to be zero if $p_Q(\beta^{(j_s)})=0$ or there is no $k<s$ with $j_k=j_s$, and otherwise to be equal to $m_k$ for the largest $k<s$ with $j_k=j_s$.
\end{defn}

\begin{rmk}
(i) The definition from \cite{CB-norm} uses  the variety  $\Rep(\Pi,\al)$ of representations of the preprojective algebra $\Pi$ of $Q$. However, the category of $\Pi$-modules is equivalent with the category of representations of the double quiver $\ol{Q}$ with zero image under the moment map by \cite[p.16]{CBH}. In particular $\Rep(\Pi,\al)=X(Q,\al)$ and the group action on both sides is the same.  Moreover, if $M$ is a representation of $\ol{Q}$ lying $X(Q,\al)$ and $M'\subset M$ is a sub-representation, say of dimension $\beta$, then $M'$ must lie in $X(Q,\beta)$. Hence the definition from above agrees with that from \cite{CB-norm}. 

(ii) The following remark is also made at the beginning of \cite[Proof of Theorem 6.3]{CB-norm}; we rephrase it slightly. Let $q:X(Q,\al)\ra M(Q,\al)$ be the affine quotient morphism. Let $x$ be a semisimple representation in $X(Q,\al)$. The points in the fiber $q^{-1}(q(x))$ correspond to representations in $X(Q,\al)$ with semisimplification isomorphic to $x$. Thus every element in $q^{-1}(q(x))$  admits a top-type with respect to the collection of simple factors of $x$. 
\end{rmk}

\begin{prop}\cite[Lemma 6.2]{CB-norm}\label{propCBL6} For a quiver $Q$ and a collection $T_1,\ldots, T_r$ of  non-isomorphic simple representations of $\ol{Q}$ lying in the zero locus of the moment maps for the corresponding dimensions $\beta^{(1)},\ldots,\beta^{(r)}$, 
the subset of $X(Q,\al)$ consisting of those elements with top-type $(j_1,m_1;\ldots ;j_h,m_h)$ with respect to $T_1,\ldots, T_r$ is constructible and has dimension at most
$$
\al\cdot \al -1 +p_Q(\al) +\sum_{s=1}^hm_sz_s -\sum_{s=1}^hm_s^2p_Q(\beta^{(j_s)}).
$$
\end{prop}

We apply this result to obtain:

\begin{prop}\label{propNil}
Let $Q$ be the quiver with one vertex and $g\ge 2$ loops. Let $n\ge 1$ be an integer.  Let $\tau=(e_i,\beta_i)_{1\le i\le r}$ be a  representation type occurring in $X(Q,n)$. Assume that $\tau\ne(1,n)$, that is, $\tau$ is the type of a semisimple but not simple representation. Let $$q_\tau:X(Q_\tau,e)\ra M(Q_\tau,e)$$ be the affine quotient morphism. Then
$$
\dim q_\tau^{-1}(0) < 2p_Q(n) - 2\sum_{i=1}^rp_Q(\beta_i).
$$
\end{prop}
\begin{proof} In this case, one can work out the right-hand side to be $2(g-1)(n^2-\sum_i\beta_i^2)$. 

{\it Step 1.} To bound the left-hand side, we will use Proposition \ref{propCBL6}. Considering the trivial representation $0$ in $X(Q_\tau,e)$, there is a decomposition 
\be\label{eqT0}
0= \oplus_{i=1}^rT_i^{\oplus e_i}\ee where $T_i$ are the simple (i.e. irreducible) non-isomorphic trivial (i.e. zero) representations of $\ol{Q}_\tau$ of dimension vectors $\eps_i$, with $\eps_i$ denoting as before the standard coordinate vectors in $\bZ^r$. In particular, $(e_i,\eps_i)_{1\le i\le r}$ is the representation type of the trivial representation in $X(Q_\tau,e)$. The fiber $q_\tau^{-1}(0)$ consists of the representations in $X(Q_\tau,e)$ with semisimplification precisely $0= \oplus_{i=1}^rT_i^{\oplus e_i}$. Hence each element in $q_\tau^{-1}(0)$ admits a top-type with respect to $T_1,\ldots, T_r$. Then, by Proposition \ref{propCBL6}, $\dim q_\tau^{-1}(0)$ is at most the maximum value among the quantities
\be\label{eqTT}
e\cdot e -1 +p_{Q_\tau}(e) +\sum_{s=1}^hm_sz_s - \sum_{s=1}^hm_s^2p_{Q_\tau}(\eps_{j_s})
\ee
for top-types $(j_1,m_1;\ldots;j_h,m_h)$ with respect to $T_1,\ldots, T_r$. Here $j_s\in\{1,\ldots,r\}$ and $m_s$ positive natural numbers, and
$$
z_s=\left\{
\begin{array}{ll}
0, & \text{ if }p_{Q_\tau}(\eps_{j_s})=0, \text{ or if }\nexists\; k<s \text{ with }j_k=j_s,\\
m_k, & \text{ for the largest }k<s \text{ with }j_k=j_s. 
\end{array}
\right.
$$
By definition of top-type, this is the data coming from an element $M\in q_\tau^{-1}(0)$ with a filtration $0=M_0\subset M_1\subset \ldots\subset M_h=M$ by sub-representations $M_s/M_{s-1}\simeq T_{j_s}^{\oplus m_s}$ and such that $\dim\homo (M_{s},T_{j_s})=m_s$ for all $s$.

{\it Step 2.} We use the special case we are in to simplify (\ref{eqTT}). Namely, $p_{Q_\tau}(e)=p_Q(n)=1+(g-1)n^2$ and $p_{Q_\tau}(\eps_i)=1+(g-1)\beta_i^2$. We plug this into (\ref{eqTT}) and we want to check that (\ref{eqTT}) is strictly smaller than $2(g-1)(n^2-\sum_i\beta_i^2)$. That is, we want that
\be\label{eqTT2}
\sum_{i=1}^re_i^2 +\sum_{s=1}^hm_sz_s-\sum_{s=1}^hm_s^2(1+(g-1)\beta_{j_s}^2) < (g-1)n^2-2(g-1)\sum_{i=1}^r\beta_i^2 .
\ee

Despite of the complicated description of the set of all possible top-types, the only information we will use now to draw our conclusion from (\ref{eqTT2}) is that the number of copies of $T_i$'s in a filtration giving rise to a top-type has to match the number of copies in the decomposition (\ref{eqT0}), namely, that for all $1\le i\le r$,
$$
e_i=\sum_{s\text{ with }j_s=i}m_s.
$$
This gives the equalities $$n=\sum_{i=1}^re_i\beta_i=\sum_{i=1}^r\sum_{s:j_s=i}m_s\beta_i=\sum_{s}m_s\beta_{j_s}$$ 
$$\sum_{i=1}^re_i^2=\sum_{i=1}^r(\sum_{s:j_s=i}m_s)^2.$$

We use these equalities to phrase (\ref{eqTT2}) only in terms of $j_s$, $m_s$, $z_s$, $\beta_i$, and $g$, eliminating the terms $n$ and $e_i$. We single out $g$ together with its coefficient, place it on the right-hand side, and  move the other summands from one side to the other until they achieve a positive sign. Then (\ref{eqTT2}) becomes equivalent to
$$
\sum_{i}(\sum_{s:j_s=i}m_s)^2 +  \sum_sm_sz_s + \sum_sm_s^2\beta_{j_s}^2 +(\sum_sm_s\beta_{j_s})^2 < $$
$$
 <\sum_sm_s^2 + 2\sum_i\beta_i^2 + g\left[ \sum_sm_s^2\beta_{j_s}^2 +(\sum_sm_s\beta_{j_s})^2 -2\sum_i\beta_i^2 \right].
$$
Since all $T_i$ must appear in the semisimplification of $M$, the set  $\{\beta_{j_s}\mid s\}$ is the whole set  $\{\beta_i\mid i\}$ with repetitions allowed. Since $m_s, \beta_i >0$, the coefficient of $g$ must be non-negative. Hence it is enough if we prove the inequality for $g=2$. That is, we need to show that
\be\label{eqTT3}
\sum_{i}(\sum_{s:j_s=i}m_s)^2 +  \sum_sm_sz_s +2\sum_i\beta_i^2 
< \sum_sm_s^2 + \sum_sm_s^2\beta_{j_s}^2+(\sum_sm_s\beta_{j_s})^2.
\ee

{\it Step 3.} By the definition of $z_s$, the inequality is equivalent with the one where the top-type is rearranged so that
$$
1=j_1=\ldots=j_{s_1};\; 2=j_{s_1+1}=\ldots=j_{s_2};\; \ldots\; ;\; r=j_{s_{r-1}+1}=\ldots=j_{s_r}=j_h.
$$
Then the sequence $z_1,\ldots,z_h$ is the sequence
$$
0,m_1,m_2,\ldots,m_{s_1-1};\; 0,m_{s_1+1},\ldots, m_{s_2-1};\; \ldots\; ;\; 0,m_{s_{r-1}+1},\ldots,m_{s_r-1}. 
$$
With this, (\ref{eqTT3}) becomes
$$
[(m_1+\ldots + m_{s_1})^2 +\ldots +(m_{s_{r-1}+1}+\ldots +m_{s_r})^2] +
$$
$$+
[(m_2m_1+m_3m_2+\ldots +m_{s_1}m_{s_1-1})+\ldots +(m_{s_{r-1}+2}m_{s_{r-1}+1}+\ldots +m_{s_r}m_{s_{r}-1}) ] +
$$
$$
+2\sum_i\beta_i^2 < [m_1^2+\ldots +m_{s_r}^2] +$$
$$
+ [\beta_1^2(m_1^2+\ldots+m_{s_1}^2)+\ldots + \beta_r^2(m_{s_{r-1}+1}^2+\ldots +m_{s_r}^2)]+
$$
$$
+[\beta_1(m_1+\ldots m_{s_1})+\ldots +\beta_r(m_{s_{r-1}+1}+\ldots +m_{s_r})]^2.
$$
By bounding the second line of summands  via
$$
m_2m_1+\ldots +m_{s_1}m_{s_1-1}\le \frac{1}{2}(m_2^2+m_1^2+\ldots +m_{s_1}^2+m_{s_1-1}^2) =
$$
$$
=m_1^2+\ldots +m_{s_1}^2 -\frac{1}{2}(m_1^2+m_{s_1}^2), 
$$
it is enough to prove that
$$
[(m_1+\ldots + m_{s_1})^2 +\ldots +(m_{s_{r-1}+1}+\ldots +m_{s_r})^2] +2\sum_i\beta_i^2 <
$$
$$
\frac{1}{2}[m_1^2+m_{s_1}^2+\ldots+m_{s_{r-1}+1}^2+m_{s_r}^2]+$$
$$+ [\beta_1^2(m_1^2+\ldots+m_{s_1}^2)+\ldots + \beta_r^2(m_{s_{r-1}+1}^2+\ldots +m_{s_r}^2)]+
$$
$$
+[\beta_1(m_1+\ldots m_{s_1})+\ldots +\beta_r(m_{s_{r-1}+1}+\ldots +m_{s_r})]^2.
$$
Defining 
$$
A_i=m_{s_{i-1}+1}+\ldots +m_{s_i},
$$
the inequality is implied by
\be\label{eqTT4}
A_1^2+\ldots +A_r^2 +\beta_1^2+\ldots +\beta_r^2 < [\beta_1A_1+\ldots +\beta_rA_r]^2
\ee
if we can show that this last one is true. Using that
$$
x^2+y^2\le x^2y^2+1\quad \text{for integers }x, y\ge 1,
$$
the left-hand side of (\ref{eqTT4}) is at most $$\beta_1^2A_1^2+\ldots +\beta_r^2A_r^2+r.$$
If $r\ge 2$, $$r\le\sum_{i\ne j}\beta_iA_i\beta_jA_j.$$ Therefore in this case
$$
A_1^2+\ldots +A_r^2 +\beta_1^2+\ldots +\beta_r^2\le \beta_1^2A_1^2+\ldots +\beta_r^2A_r^2+\sum_{i\ne j}\beta_iA_i\beta_jA_j =
$$ 
$$
=[\beta_1A_1+\ldots +\beta_rA_r]^2,
$$
which proves (\ref{eqTT4}) for this case. If $r=1$, (\ref{eqTT4}) reduces to $A_1^2+\beta_1^2<\beta_1^2A_1^2$. By assumption, $A_1, \beta_1> 1$, so the inequality  holds in this case as well.
\end{proof}

\begin{prop}\label{propZBD}
Let $Q$ be the quiver with one vertex and $g\ge 2$ loops. Let $n\ge 1$ be an integer.  Let $\tau=(e_i,\beta_i)_{1\le i\le r}$ be a  representation type occurring in $X(Q,n)$ with $\tau\ne(1,n)$, that is, $\tau$ is the representation type of a semisimple but not simple representation. Then
$$
\dim Z(Q_\tau,e) < 2p_Q(n).
$$
\end{prop}
\begin{proof} 
It follows from Lemma \ref{propEXX} and Proposition \ref{propNil}, since $\sum_{i=1}^rp_{Q}(\beta_i)$ is the number of loops in $Q_\tau$ by definition of $Q_\tau$.
%It is clear from the proof that Proposition \ref{propNil} also holds for $q_\tau^{-1}(q_\tau(x))$ for any point $x\in F(Q_\tau,e)$, since all points in the fixed locus have the same representation type as the zero representation. By  Lemma \ref{lemFix}, the dimension of $q_\tau(F(Q_\tau,e))$ is $\sum_{i=1}^r2p_{Q}(\beta_i)$, twice the number of loops in $Q_\tau$. Hence $$\dim Z(Q_\tau,e)\le \sum_{i=1}^r2p_{Q}(\beta_i) + \dim q_\tau^{-1}(0).$$ We now use the bound from Proposition \ref{propNil} to obtain that $$\dim Z(Q_\tau,e)< 2p_Q(n)=2p_{Q_\tau}(e).$$
\end{proof}

\begin{rmk}  The inequality does not hold for $\tau=(1,n)$. Indeed, in this case $Z(Q_\tau,e)$ consists of the simple locus in $X(Q,n)$, which by Theorem \ref{thrmCBd} is open and of dimension $n^2-1+2p_Q(n)$.
\end{rmk}

Proposition \ref{propZBD} has the following more explicit equivalent formulation:

\begin{prop}\label{propEZBD}
 Let $\mathscr{C}$ be the class of quivers parametrized up to the direction of the arrows by the set $$\bigcup_{r>0} \bZ_{\ge 2}\times\bZ_{>0}^{r},$$ such that a quiver $Q$ associated to $(g, \beta)\in \bZ_{\ge 2}\times\bZ_{>0}^{r}$ has: $r$ vertices, $1+(g-1)\beta_i^2$ loops at the vertex $i$, and $(g-1)\beta_i\beta_j$ arrows between the vertices $i$ and $j$ with $i\ne j$, there being no restriction on the direction of these arrows. Then, for every $\al\in\bZ_{>0}^r$, excluding the case $r=\al=1$,
$$
\dim Z(Q,\al)< 2p_Q(\al).
$$
\end{prop}
\begin{proof} Let $Q_g$ be the quiver with 1 vertex and $g>1$ loops. For an integer $m>1$, $p_Q(m)=1+(g-1)m^2>0$. So by Theorem \ref{thrmCBd}, $X(Q_g,m)$ has lots of simple representations. By taking direct sums of simple representations of various dimensions, one has that an equivalent condition for $\tau=(\al_i,\beta_i)_{1\le i\le r}$ to be a representation type occurring in $X(Q_g,n)$ is that $\sum_{i=1}^r{\al_i\beta_i}=n$.
One applies now Proposition \ref{propZBD} with $\al_i=e_i$ and the explicit description in this case of $Q_\tau$ from Definition \ref{defCB}. \end{proof}

\section{Jets along zeros of moment maps}\label{s3}

We fix as before an algebraically closed field $k$ of characteristic zero. For $X$ a variety over $k$ and an integer $m\ge 1$, we let $$\pi_m:X_m\ra X$$ denote the projection from the $m$-jet scheme. Recall that the $m$-jets are the elements of $$X_m(k)=\homo_{k-\rm{sch}}(\spec(k[t]/t^{m+1}),X).$$
If $X\subset \bA^n$ is given by $f_1(x)=\ldots=f_r(x)=0$ with $x=(x_1,\ldots,x_n)$ and $f_i\in k[x]$, then $X_m$ is given in $\bA^{n(m+1)}$ with coordinates $x,x',x'',\ldots,x^{(m)}$, where $x^{(j)}=(x^{(j)}_1,\ldots,x_n^{(j)})$, by 

$$
f_1(x(t))\equiv \ldots \equiv f_r(x(t))\equiv 0 \mod t^{m+1}
$$
with $x_i(t)=x_i+x'_it+x''_it+\ldots +x_i^{(m)}t^m$.

We will use the following relationship between jets schemes and rational singularities due to M. Musta\c{t}\u{a}:

\begin{thrm}\label{thrmMus}{\rm{}(}\cite[Propositions 1.4 and 1.5]{Mus}{\rm{)}}
Let $X$ be a locally complete intersection variety. The following are equivalent for $m\ge 1$:

(i) $X_m$ is irreducible,

(ii) $\dim \pi_m^{-1}(X_{sing})<(\dim X)(m+1)$,

(iii) $X_m$ is a locally complete intersection variety of dimension $\le (\dim X)(m+1)$.
\end{thrm}

\begin{thrm}\label{thrmMus2} {\rm{}(}\cite[Theorems 0.1 and 3.3]{Mus}{\rm{)}}
Let $X$ be a locally complete intersection variety. The following are also equivalent:

(a) the conditions (i)-(iii) are fulfilled for all $m\ge 1$,

(b) $X$ has rational singularities,

(c) $X$ has canonical singularities.

\end{thrm}

\begin{lemma}\label{lemFixJet} For a quiver $Q$ and dimension vector $\al$,
let $\pi_m:X(Q,\al)_m\ra X(Q,\al)$ be the projection from the $m$-jet scheme, $m\ge 1$. Then 
$$\pi_m^{-1} (F(Q,\al))\simeq\left\{
\begin{array}{ll}
F(Q,\al)\times X(Q,\al)_{m-2}\times \Rep(\ol{Q},\al) & \text{ if }m\ge 2,\\
F(Q,\al)\times \Rep(\ol{Q},\al) & \text{ if }m=1.
\end{array}\right.
$$
\end{lemma}
\begin{proof} As before, $F(Q,\al)$ denotes the locus in $X(Q,\al)$ fixed by the action of $G(\al)$. The $m$-jet scheme $X(Q,\al)_m$ is the subscheme of
$$
\bigoplus_{a\in Q}\left( \Mat (\al_{h(a)}\times \al_{t(a)},k[t]/t^{m+1})  \oplus  \Mat (\al_{h(a^*)}\times \al_{t(a^*)},k[t]/t^{m+1})\right)
$$
defined by 
$$
\sum_{a\in Q}[x(t)_a,x(t)_{a^*}]=0 \in \bigoplus_{i\in I}\Mat(\al_i, k[t]/t^{m+1}).
$$
Writing $x(t)_a=\sum_{j=0}^m x_{a,j}t^j$ with $x_{a,j}$ a matrix over $k$, this is equivalent with the system of equations
$$
\sum_{a\in Q}[x_{a,0},x_{a^*,0}] =0$$
$$
\sum_{a\in Q}\left([x_{a,0},x_{a^*,1}] + [x_{a,1},x_{a^*,0}] \right) =0
$$
$$
\sum_{a\in Q}\left([x_{a,0},x_{a^*,2}] + [x_{a,1},x_{a^*,1}] + [x_{a,2},x_{a^*,0}]\right) =0
$$
$$
\cdots
$$
$$
\sum_{a\in Q}\left([x_{a,0},x_{a^*,m}] + [x_{a,1},x_{a^*,m-1}] +\ldots + [x_{a,m-1},x_{a^*,1}] + [x_{a,m},x_{a^*,0}]\right) =0.
$$
If $x(t)=(x(t)_a,x(t)_{a^*})_{a\in Q}$ is an $m$-jet in $X(Q,\al)_m$ whose projection is the fixed point $$x_0=\pi_m(x(t))=(x_{a,0},x_{a^*,0})_{a\in Q},$$ then all the commutator terms involving $x_{a,0}$ or $x_{a^*,0}$ are zero, by Lemma \ref{lemFix}. The conclusion follows. 
\end{proof}

As before, let $Z(Q,\al)$ denote the subset of points in $X(Q,\al)$ with $G(\al)$-orbit closure containing a $G(\al)$-fixed point.

\begin{lemma}\label{lemRSA}
 Let $Q$ be a quiver. Let $\al$ be a dimension vector with $p_Q(\al)>0$.  Assume that { $X(Q,\al)$ contains a simple representation, and that} for all semisimple but not simple representation types $\tau=(e,\beta)$ occurring in $X(Q,\al)$,  
 $$\dim Z(Q_\tau,e)< 2p_Q(\al).$$
Then $X(Q,\al)$ has rational singularities. 
\end{lemma}
\begin{proof} {For this proof we will use $X=X(Q,\al)$, $M=M(Q,\al)$, and $G=G(\al)$ to ease the notation.}

By Theorem \ref{thrmCBd}, $X$ is a complete intersection variety, smooth at the simple representations. Therefore we can use Theorem \ref{thrmMus2} to show that $X$ has rational singularities. Namely, letting $\pi_m:X_m\ra X$ be the projection from the $m$-th jet scheme, we will show that 
\be\label{eqS}\dim \pi_m^{-1}(X_{sing})<(\dim X)(m+1)\ee
for all $m\ge 1$.

Since $X_{sing}$ is contained in the union over all non-simple representation types $\tau=(e,\beta)$ occurring in $X$ of the strata $X_\tau$ from Definition \ref{defnStr}, it is enough to show that 
$$
\dim \pi_m^{-1}(X_\tau)<(\dim X)(m+1)
$$
for all $m\ge 1$.

Let $x\in X_\tau$ be a semisimple representation. Then  we have the \'etale-local isomorphisms 
\be\label{eqISG}
(X,x)\simeq (Gx,x)\times (X(Q_\tau,e),0)
\ee
over saturated open affine neighborhoods restricting to an \'etale-local isomorphism
$$
(X_\tau,x)\simeq (Gx,x)\times (Z(Q_\tau,e),0)
$$
by Proposition \ref{propStrat}. Recall that the formation of $m$-jets commutes with taking \'etale covers, see \cite[Proposition 1.1]{Mus}. Hence there is an \'etale-local isomorphism 
\be\label{eqXtau}
\pi_m^{-1}(X_\tau) \simeq (Gx)_m \times \pi_{\tau,m}^{-1}(Z(Q_\tau,e))
\ee
above the saturated open affine neighborhoods from (\ref{eqISG}), where $(Gx)_m$ is the $m$-jet scheme of the orbit $Gx$, and $\pi_{\tau,m}:X(Q_\tau,e)_m\ra X(Q_\tau,e)$ is the projection from the $m$-jet scheme of $X(Q_\tau,e)$. We stress that in (\ref{eqXtau})  the \'etale-local isomorphism might not even be defined globally on some \'etale covers.

We will compute dimensions in (\ref{eqXtau}). In doing so, we do not have to restrict above a saturated affine neighborhood of $x\in X$ since $X_\tau$ is saturated by definition, and  there is an \'etale-local isomorphism $(X,x)\simeq (X,x')$ over saturated affine neighborhoods for any $x'\in X_\tau$ semisimple by (\ref{eqISG}). Similarly, in computing the dimension on the right-hand side, we do not have to restrict above a saturated affine neighborhood of $0$ in $X(Q_\tau,e)$.  

Since $Gx$ is smooth irreducible, $$\dim (Gx)_m= (\dim Gx)(m+1).$$ 
Thus,
$$
\dim \pi_m^{-1}(X_\tau) = (\dim Gx)(m+1)+\dim \pi_{\tau,m}^{-1}(Z(Q_\tau,e)).
$$

Since $
(X(Q_\tau,e), y)\simeq (X(Q_\tau,e),0)
$ for all semisimple $y\in Z(Q_\tau,e)$, we have $\dim \pi_{\tau,m}^{-1}(y)=\dim \pi_{\tau,m}^{-1}(0)$. Since the dimensions of the fibers of $\pi_{\tau,m}$ obey upper semi-continuity, see \cite[Proposition 2.3]{Mu2}, we have then
$$
\dim \pi_m^{-1}(X_\tau) \le (\dim Gx)(m+1) + \dim Z(Q_\tau,e)+\dim \pi_{\tau,m}^{-1}(0).
$$
By assumption,
$$
\dim Z(Q_\tau,e)<2p_Q(\al) = \dim M,
$$
the last equality being true by Theorem \ref{thrmCBd}. Hence
\be\label{eqImp}
\dim \pi_m^{-1}(X_\tau) < (\dim Gx)(m+1) +\dim M+\dim \pi_{\tau,m}^{-1}(0).
\ee

Consider the case $m=1$. By Lemma \ref{lemFixJet}, 
$$\dim \pi^{-1}_{\tau,m}(0)=\dim \Rep(\ol{Q_\tau},e)=2(e\cdot e-1+p_{Q_\tau}(e)).$$
We have $\dim G_x=\dim G(e) = e\cdot e -1$, and $2p_{Q_\tau}(e)=2p_Q(\al)=\dim M$ by Proposition \ref{propSim}. Hence
$$
\dim \pi^{-1}_{\tau,1}(0)= 2\dim G_x + \dim M,
$$
and so (\ref{eqImp}) becomes
$$
\dim\pi_1^{-1}(X_\tau) < 2(\dim Gx +\dim M + \dim G_x) = 2(\dim G +\dim M)=2\dim X,
$$
the last equality  being true by Theorem \ref{thrmCBd}. This proves the case $m=1$ of (\ref{eqS}).

Let now $m\ge 2$. By Lemma \ref{lemFixJet}, 
\be\label{eqANA}
\dim \pi^{-1}_{\tau,m}(0)=\dim X(Q_\tau,e)_{m-2}+\dim \Rep(\ol{Q_\tau},e).
\ee

Every representation type that occurs in $X(Q_\tau,e)$ actually occurs in any $G(e)$-saturated open affine neighborhood of $0\in X(Q_\tau,e)$, cf. the first paragraph of the proof of Proposition \ref{prop1T1}. By using the group action, every point in $X(Q_\tau,e)$ outside the semisimple locus, has an open neighborhood isomorphic to an open neighborhood of a point arbitrarily close to the semisimple locus. Therefore the dimension of $X(Q_\tau,e)_{m-2}$ can be computed above any saturated open affine neighborhood of $0\in X(Q_\tau,e)$. Hence, we can use
 (\ref{eqISG}), which gives us that
$$
\dim X(Q_\tau,e)_{m-2} \le \dim X_{m-2} - \dim (Gx)_{m-2} =\dim X_{m-2} - (\dim Gx)(m-1).
$$

By induction on $m$, we can assume that $$\dim \pi_{m-2}^{-1}(X_\tau)< (\dim X)(m-1)$$ for all semisimple but not simple representation type $\tau$ occurring in $X$. Since $$\dim \pi_{m-2}^{-1}(X\setminus X_{sing})=(\dim X)(m-1),$$ we have $$\dim X_{m-2}\le (\dim X)(m-1).$$ Therefore
$$
\dim X(Q_\tau,e)_{m-2}\le (\dim X-\dim Gx)(m-1).
$$
With this and  (\ref{eqANA}),  we get from (\ref{eqImp}) that
$$
\dim \pi_m^{-1}(X_\tau) < (\dim Gx)(m+1) +\dim M + (\dim X - \dim Gx)(m-1)+$$
$$ + 2\dim G_x +\dim M = (\dim G_x +\dim X)(m+1)=(\dim X)(m+1).
$$
\end{proof}

\begin{prop}
Let $Q$ be a quiver with a dimension vector $\al$ such that $p_Q(\al)>0$.  Assume that { $X(Q,\al)$ contains a simple representation, and that} for all semisimple but not simple representation types $\tau=(e,\beta)$ occurring in $X(Q,\al)$, the inequality (\ref{eqClass}) from Lemma \ref{propEXX} holds for $(Q_\tau,e)$. Then $X(Q_\tau,e)$ has rational singularities for all representation types $\tau=(e,\beta)$ occurring in $X(Q,\al)$.
\end{prop}
\begin{proof} By Proposition \ref{propSim}, Proposition \ref{prop1T1}, Lemma \ref{propEXX}, we can apply Lemma \ref{lemRSA} to $X(Q_\tau,e)$.
\end{proof}

A direct consequence of this is:

\begin{thrm}\label{thrmEXT}
Let $\mathscr{M}$ be a class of quivers with dimension vectors $(Q,\al)$ satisfying:
\begin{itemize}
\item $\mathscr{M}$ is closed under the operation of producing quivers with dimension vectors $(Q_\tau,e)$ via Definition \ref{defCB}; and
%\item For every $(Q,\al)$ in $\mathscr{M}$ not generated via Definition \ref{defCB} from another element of $\mathscr{M}$ (equivalently, for all $(Q,\al)\in \mathscr{M}$), the inequality $p_Q(\al)>0$ holds; and
\item For every $(Q,\al)$ in $\mathscr{M}$, $X(Q,\al)$ contains a simple representation and $p_Q(\al)>0$; and
\item For every $(Q,\al)$ in $\mathscr{M}$ such that $X(Q,\al)$ contains semisimple but not simple representations, the statement (\ref{eqClass}) holds. That is, the dimension of the union of all $G(\al)$-orbits in $X(Q,\al)$ whose closures contain the trivial representation is strictly smaller than
$$
2p_Q(\al) - 2\cdot\#(\text{loops in } Q).
$$
\end{itemize}
Then $X(Q,\al)$ has rational singularities for every $(Q,\al)$ in $\mathscr{M}$.
\end{thrm}
%\begin{proof} Let $(Q,\al)$ be in $\mathscr{M}$. Let $Z=Z(Q,\al)$, $F=F(Q,\al)$, $G=G(\al)$. We can assume that $X$ has semisimple but not simple representations, otherwise $X$ is smooth. Then $\dim Z\le 2p_Q(\al)$ by the assumption and Lemma \ref{propEXX}. By assumption, $X(Q_\tau,e)$ is in $\mathscr{M}$ for every semisimple but non-simple representation type $\tau=(e,\beta)$ occurring in $X$, and hence $\dim Z(Q_\tau,e)<2p_{Q_\tau}(e)$. However, $p_{Q_\tau}(e)=p_Q(\al)$ by Proposition \ref{propSim}. Now we can apply Lemma \ref{lemRSA} to conclude that $X$ has rational singularities.\end{proof}

\begin{thrm}\label{thrmExtend}
For every quiver $Q$ in $\mathscr{C}$, the class of quivers defined in Proposition \ref{propEZBD}, and for every dimension vector $\al\in \bZ_{>0}^r$ on the set of vertices of $Q$ with non-zero dimension at each vertex, the variety $X(Q,\al)$ has rational singularities.
\end{thrm}
\begin{proof}
Consider the class $\mathscr{M}$ of quivers with dimension vectors $(Q,\al)$ such that $Q$ is in $\mathscr{C}$ and $\al\in \bZ_{>0}^r$. By definition, $\mathscr{M}$ the class obtained by applying Definition \ref{defCB} to quivers with one vertex and at least two loops. By Proposition \ref{prop1T1}, $\mathscr{M}$ is closed under the operation of producing new elements via Definition \ref{defCB}. Moreover, (\ref{eqClass}) holds for all $(Q,\al)$ in $\mathscr{M}$ by Proposition \ref{propNil}. The conclusion follows from Theorem \ref{thrmEXT}.
\end{proof}

\noindent{\it Proof of Theorem \ref{thrmRSQ}.} The quiver with one vertex and $g\ge 2$ loops lies by definition in the class $\mathscr{C}$. The conclusion follows from Theorem \ref{thrmExtend}. $\hfill\Box$

\section{$\GL_n$-representations of surface groups}\label{s4}

In this section $k=\bC$. For a compact Riemann surface $C_g$ of genus $g$ and a point $p\in C_g$ we will use the simpler notation
$$
R(g,n) = \homo(\pi_1(C_g,p), \GL_n(k))
$$
and
$$
M(g,n)=R(g,n)\sslash \GL_n(k)
$$
for the affine quotient. The closed points of $M(g,n)$ are into 1-1 correspondence with the isomorphism classes of semisimple $k$-local systems of rank $n$ on $C_g$. For $g\ge 2$, $R(g,n)$ is a normal complete intersection variety by \cite[Theorem 11.1]{Si}, while $M(g,n)$ is a symplectic variety by \cite[Proposition 8.4]{BS}, hence it has rational singularities.

We need the following from deformation theory, see for example the case $k=0$ of \cite[Lemma 7.10]{BW}:

\begin{prop}\label{propDEF}  Let $g\ge 2$. Let $\rho\in R(g,n)$ be a semisimple complex representation and $L\in M(g,n)$ the associated semisimple complex local system. Let 
$$
\mu:\Ext^1(L,L)\lra \Ext^2(L,L),\quad \mu(e)=e\cup e.
$$
There is an isomorphism of formal germs
$$
\widehat{R(g,n)}_\rho\simeq \widehat{\frak{h}}_0\times \widehat{\mu^{-1}(0)}_0,
$$
for some affine space $\mathfrak{h}$.
\end{prop}

\begin{rmk}
(a) Recall that
$
\Ext^\ubul(L,L)=H^\ubul (C_g, L\otimes L^\vee).
$

(b) For two points $x\in X$ and $y\in Y$ on two algebraic varieties, the following are equivalent by Artin approximation: $x$ and $y$ have isomorphic complex analytic, or \'etale, or formal neighborhoods.

(c) The main ingredient of the proof of Proposition \ref{propDEF} is the formality of the differential graded Lie algebra controlling the deformation theory of a semisimple representation. This was proved by \cite{Si-form} using transcendental techniques, hence the reason for the field $k$ to be $\bC$.
\end{rmk}

%\begin{rmk} The above discussion is a substitute for \cite[2.6]{KLS} in the context of coherent sheaves. We will only be using part (b).  \end{rmk}

The map $\mu$ is the moment map of a double quiver, see \cite[3.1]{KLS}, where one has to replace Serre duality for coherent sheaves with duality for local systems \cite[Corollary 3.3.12]{Di}, or see  \cite[Theorem 8.6]{BS} for this version with local systems which we state now:

\begin{prop}\label{propL} Let  $$L\simeq L_1^{\oplus e_1}\oplus\cdots\oplus L_r^{\oplus e_r}$$ be a semisimple complex local system of representation type $$\tau=(e_1,\beta_1;\cdots;e_r,\beta_r).$$ on $C_g$, with $\rank(L_i)=\beta_i$. Let $Q$  be the quiver with one vertex and $g$ loops. Consider the associated quiver $Q_\tau$, dimension vector $e$, and group $G(e)$ as is Definition \ref{defCB}. Then:

(a) $\Aut(L)/k^*\simeq G(e)$.

(b) There is a commutative diagram
$$
\xymatrix{
\Ext^1(L,L)  \ar[rr]^\sim \ar[d]_{\mu}& &\Rep(\ol{Q_\tau},e)\ar[d]^{\mu_{Q_\tau,e}}\\
\Ext^2(L,L) \ar[rr]^\sim& & \ehom(e)
}
$$
where the horizontal maps are $G(e)$-equivariant isomorphisms.

(c) There is a $G(e)$-equivariant isomorphism
$$
\mu^{-1}(0) \simeq X(Q_\tau,e).
$$
\end{prop}
\begin{proof}
The quiver realizing this identification from \cite{KLS, BS} is described as having $r$ vertices and its double having $\dim \Ext^1(L_i,L_j)$ arrows from $i$ to $j$. The identification with $Q_\tau$ follows from Definition \ref{defCB} and the equalities
$$
\dim \Ext^1(L_i,L_j)=\left\{
\begin{array}{cll}
2p_{Q}(\beta_i)& =2(g-1)\beta_i^2+2, & \text{ if }i=j,\\
-(\beta_i,\beta_j)_{Q} & =2(g-1)\beta_i\beta_j, & \text{ if }i\ne j.
\end{array}
\right.
$$
We recall how one checks these equalities. One has
$$
\chi(C_g, L_i^\vee\otimes L_j)=\rank(L_i^\vee\otimes L_j)\chi(C_g)=\beta_i\beta_j2(1-g)
$$
by \cite[Example 3.3.13]{Di}. On the other hand, by duality $$h^2(C_g, L_i^\vee\otimes L_j)=h^0(C_g, L_i\otimes L_i^\vee)=\dim\homo(L_j,L_i).$$ This equals 1 if $i=j$, and it equals 0 if $i\ne j$ since the $L_i$ are simple non-isomorphic local systems. Hence $h^1(C_g, L_i^\vee\otimes L_j)=\dim\Ext^1(L_i,L_j)$ has the claimed value.
\end{proof}

The previous two propositions imply:

\begin{cor}\label{thrmXQg}
Let $g\ge 2$. Let $\rho\in R(g,n)$ be a semisimple complex representation of type $\tau=(e_1,\beta_1;\ldots;e_r,\beta_r)$. Let $Q$ be the quiver with one vertex and $g$ loops. There is an isomorphism of formal germs
$$
\widehat{R(g,n)}_\rho\simeq \widehat{\mathfrak{h}}_0\times \widehat{X(Q_\tau,e)}_0
$$
for some affine space $\mathfrak{h}$.
\end{cor}

\begin{rmk}
In \cite[\S 3]{KLS}, the  analog of this corollary for coherent sheaves on K3 surfaces was proved with some difficulty, see \cite[Propositions 3.6 and 3.8]{KLS}, since at the time, formality for polystable sheaves on K3 surfaces was not known. This formality property has been proven recently by Budur-Zhang \cite{BZ}. 
\end{rmk}

\noindent {\it Proof of Theorem \ref{thrmGLRat}.} 
The moduli space  $R(g,n)$ contains an open dense subsets of isomorphism classes of simple representations. In particular, by taking direct sums,  $R(g,n)$ contains semisimple representations of type $\tau=(e_1,\beta_1;\ldots; e_r,\beta_r)$ for every partition $\sum_{i=1}^re_ib_i=n$ with $r, e_i, \beta_i\in \bN\setminus\{0\}$. These are also all the representation types occurring on $X(Q,n)$, where $Q$ is the quiver with one vertex and $g$ loops. The proof of Theorem \ref{thrmRSQ} gives that all $X(Q_\tau,e)$ have rational singularities. By Corollary \ref{thrmXQg}, this implies that all semisimple representations in $R(g,n)$ are rational singularities. So every semisimple representation in $R(g,n)$ has an open neighborhood which has rational singularities. Using the group action to map isomorphically a neighborhood of any point to a neighborhood of a point arbitrarily close to a semisimple point, we conclude that $R(g,n)$ has rational singularities everywhere. $\hfill\Box$

\section{$\SL_n$-representations of surface groups}\label{s5}

The sets
$
\homo(\pi_1(C_g),\GL_n(\bC))$ and $\homo(\pi_1(C_g),\SL_n(\bC))
$
are the sets of complex points of two schemes of finite type over $\bQ$ which we denote by
$
R(g,\GL_n)$ and $R(g,\SL_n),
$
respectively. 
The connection between them is given by the following, see \cite[Lemma 8.17]{BS} for the complex points: 

\begin{lemma}\label{lemETA}
Let $g, n\ge 1$. There is an \'etale morphism of $\bQ$-schemes
$$
R(g,\SL_n)\times_\bQ (\bG_m)^{2g}\ra R(g,\GL_n),
$$
where $\bG_m=\spec \bQ[T,T^{-1}]$.
\end{lemma}
\begin{proof} On the underlying sets of complex points, this just the morphism 
$$
\homo(\pi_1(C_g),\SL_n(\bC))\times (\bC^*)^{2g} \lra \homo(\pi_1(C_g),\GL_n(\bC))
$$
$$
((x_i,y_i)_{1\le i\le g}, (\lam_i,\mu_i)_{1\le i\le g})\mapsto (\lam_ix_i,\mu_iy_i)_{1\le i\le g}.
$$
To prove the claim over $\bQ$, we need to work with  $\bQ$-algebras though. Let $$A=\bQ[s_1^{\pm 1}, t_1^{\pm 1},\ldots,s_g^{\pm 1}, t_g^{\pm 1}]$$ be ring of regular functions on $(\mathbb{G}_m)^{2g}$. Define the morphism of $\bQ$-algebras $$m^\#:A\ra A,\quad (s_i,t_i)\mapsto (s_i^n,t_i^n).$$
Let  $m: (\mathbb{G}_m)^{2g}\ra (\mathbb{G}_m)^{2g}$ be the corresponding morphism on spectra. Then $m$ is an \'etale morphism sending a closed point $(\lambda_i,\mu_i)_{1\le i\le g}$ to $(\lambda_i^n,\mu_i^n)_{1\le i\le g}$.  

Let $B$ be the ring of regular functions on $\R{g}{\bGL_n}$, that is, 
$$
B=\frac{\bQ[X_i,Y_i,\det^{-1}(X_i),\det^{-1}(Y_i)]_{1\le i\le g}}{\langle \prod_{i=1}^g[X_i,Y_i]-1\rangle}
$$
where
$X_i=((X_i)_{jk})_{1\le j,k\le n}$, $Y_i=((Y_i)_{jk})_{1\le j,k\le n}$ are matrices of indeterminates. Define the $\bQ$-algebra morphism
$$
f^\#:A\ra B,\quad (s_i,t_i)\mapsto (\det(X_i),\det(Y_i))
$$
Let $f:\R{g}{\GL_n}\ra (\mathbb{G}_m)^{2g}$ be the corresponding morphism of $\bQ$-schemes. It is given 
by $$f:(g_i,h_i)_{1\le i\le g}\mapsto(\det(g_i),\det(h_i))_{1\le i\le g}$$
on closed points.  Form the fiber product
$$
\xymatrix{
\R{g}{\GL_n} \times _{(\mathbb{G}_m)^{2g}}(\mathbb{G}_m)^{2g} \ar[d]_{f'} \ar[r]^{\quad \quad \quad m'}& \R{g}{\GL_n} \ar[d]_{f}\\
(\mathbb{G}_m)^{2g} \ar[r]^m& (\mathbb{G}_m)^{2g}.
}
$$
We claim that there is an isomorphism of $(\mathbb{G}_m)^{2g}$-schemes
$$
\xymatrix{
\R{g}{\SL_n} \times _{\bQ}(\mathbb{G}_m)^{2g} \ar[rr]_g^\sim \ar[dr]_{p}& & \R{g}{\GL_n} \times _{(\mathbb{G}_m)^{2g}}(\mathbb{G}_m)^{2g} \ar[dl]^{f'} \\
& (\mathbb{G}_m)^{2g} &
}
$$
where $p$ is the second projection. On closed points, $g$ sends a  point $((g_i,h_i),(\lambda_i, \mu_i))$ of $\R{g}{\SL_n} \times _{\bQ}(\mathbb{G}_m)^{2g}$ to the point $((\lambda_ig_i,\mu_ih_i),(\lambda_i, \mu_i))$ of $\R{g}{\GL_n} \times _{(\mathbb{G}_m)^{2g}}(\mathbb{G}_m)^{2g}$. 

To define $g$ schematically, we define  the corresponding $\bQ$-algebra morphism $g^\#$.  Let 
$$
C=\frac{B}{\langle \det(X_i)-1, \det(Y_i)-1 \rangle_{1\le i\le g}}\otimes_\bQ A.
$$
Then $$\R{g}{\SL_n} \times _{\bQ}(\mathbb{G}_m)^{2g}=\spec(C).$$
Let $D=
B\otimes_ A A
$
via the $A$-algebra morphisms $f^\#:A\ra B$ and $m^\#:A\ra A$. Then
$$\R{g}{\GL_n} \times _{(\mathbb{G}_m)^{2g}}(\mathbb{G}_m)^{2g}=\spec(D).$$ 
Note that there is a natural isomorphism of $A$-algebras
$$
D\simeq \frac{B\otimes_\bQ A}{\langle \det(X_i)-s_i^n, \det(Y_i)-t_i^n\rangle}
$$
Define first a morphisms of $A$-algebras 
$$g^\#:B\otimes_\bQ A\ra B\otimes_\bQ A$$
by $X_i \mapsto s_iX_i$ (that is, with $s_iX_i$ viewed as a matrix), $Y_i \mapsto t_iY_i$, $s_i\mapsto s_i$, $t_i\mapsto t_i$. This is clearly an isomorphism over $A$, since $s_i$ and $t_i$ are invertible and the inverse is given by $X_i\mapsto t_i^{-1}X_i$, $Y_i\mapsto t_i^{-1}Y_i$. Moreover, $g^\#(\det(X_i)-s_i^n)=s_i^n(\det(X_i)-1)$, and similarly for $Y_i$. Hence the ideal generated by the image under $g^\#$ of the ideal
$\langle \det(X_i)-s_i^n, \det(Y_i)-t_i^n\rangle  
$
is precisely the ideal $\langle \det(X_i)-1, \det(Y_i)-1 \rangle$, since $s_i$ and $t_i$ are invertible. It follows that $g^\#$ induces a well-defined $A$-algebra isomorphism $
g^\#: D \xa{\sim} C.
$
Thus $g$ is an isomorphism over $(\bG_m)^{2g}$.

The base change of an \'etale morphism is \'etale. Hence $m'$ is \'etale. Then the composition
$$
m'\circ g:\R{n}{\SL_d} \times _{k}(\mathbb{G}_m)^{2n} \ra \R{n}{\GL_d}
$$
is \'etale since $g$ is an isomorphism.
\end{proof}

\smallskip
\noindent{{\it Proof of Theorem \ref{thrmSLD}}.} By general properties of schemes (see for example the proof of \cite[Lemma 3.2]{BZ}), Theorem \ref{thrmGLRat} implies that $R(g,\GL_n)$ is a $\bQ$-variety with rational singularities for $g\ge 2$. By Lemma \ref{lemETA}, the same is true for $R(g,\SL_n)$. This implies in particular Theorem \ref{thrmSLD} and it shows that the natural scheme structure on $\homo(\pi_1(C_g),\SL_n)$ is reduced for $g\ge 2$.  $\hfill\Box$

\section{Consequences}\label{secAAG}

We recall the following terminology from Theorem \ref{thrmGAM1}. Let $\Gamma$ be a topological group. The abscissa of convergence $$\al(\Gamma)$$ is the smallest $s_0\in\bR\cup\{\infty\}$ such that representation zeta function
$$
\zeta_\Gamma (s)=\sum_{m\ge 1}r_m(\Gamma)m^{-s},
$$
with 
$r_m(\Gamma)
$ being the number of isomorphism classes of continuous irreducible $m$-dimensional complex representations of $\Gamma$, converges for $Re(s)>s_0$. 

In Theorem \ref{thrmGAM3}, the abscissa $\al$ is defined similarly but counting all the irreducible representations without restricting to the continuous ones. 

In order for $\zeta_\Gamma(s)$ to be defined, $r_m(\Gamma)$ must  be finite. For finitely generated pro-finite groups $\Gamma$, this property is equivalent to having finite abelianization by \cite{JK}. This is the case for the groups in Theorems \ref{thrmGAM2} and \ref{thrmGAM1}. For arithmetic groups in higer rank semisimple groups, the finiteness of $r_m(\Gamma)$ is a consequence of Margulis superrigidity, see \cite[p2]{BLMM}. This is the case for the groups in Theorems \ref{thrmGAM0} and \ref{thrmGAM3}.

\smallskip
\noindent{\it Proof of Theorem \ref{thrmGAM2}.} Follows from Theorem \ref{thrmGAM1} and general facts about abscissae of convergence of Dirichlet generating functions. $\hfill\Box$

\smallskip
\noindent{\it Proof of Theorems \ref{thrmGAM1} and \ref{thrmGAM4}.} They follow directly from Theorem \ref{thrmSLD} and \cite[Theorem IV]{AA}. $\hfill\Box$

\smallskip
\noindent{\it Proof of Theorem \ref{thrmGAM3}.} Follows from Theorem \ref{thrmSLD} together with \cite[Theorem II]{AA1}. See the comment following Corollary 1.4 in \cite{BZo}. $\hfill\Box$

\smallskip
\noindent{\it Proof of Theorem \ref{thrmGAM0}.} Follows from Theorem \ref{thrmGAM3} and general facts about abscissae of convergence of Dirichlet generating functions. $\hfill\Box$

\end{document}